\documentclass[11pt]{article}

\usepackage{amssymb,latexsym}
\usepackage{amsmath,amscd}
\usepackage{theorem}
\usepackage{url}
\usepackage{bm} 
\usepackage{graphics}
\usepackage{color}
\usepackage{soul}

\title{\bf On the self-similarity index\\ of $\bm{p}$-adic analytic 
pro-$\bm{p}$ groups}

\author{
        Francesco Noseda  
        \\[0.1cm]
        Ilir Snopce \thanks{Supported by CNPq,
        FAPERJ and the Alexander von Humboldt Foundation.}
}

\date{}

\newcommand{\bb}[1]{\mathbb{#1}}

\newcommand{\mr}[1]{\mathrm{#1}}

\newcommand{\rar}{\rightarrow}

\newcommand{\gen}[1]{\langle #1 \rangle}

\newcommand{\vep}{\varepsilon}

\newcommand{\les}{\leqslant}
\newcommand{\ges}{\geqslant}

\newcommand{\sylow}{\scalebox{0.5}{\mbox{\ensuremath{\displaystyle \triangle}}}}

\newcommand{\bs}{\backslash}
\newcommand{\ep}{\hfill $\square$} 

\newtheorem{lemma} {Lemma} 

\newtheorem{theorem} [lemma] {Theorem}
\newtheorem{corollary} [lemma] {Corollary}

\theorembodyfont{\rm} 

\newtheorem{example}[lemma] {Example}
\newtheorem{remark}[lemma]{Remark}

\newenvironment{proof}{{\sc Proof:}}{
\hfill $\square$}

\numberwithin{equation}{section}

\theorembodyfont{\it}

\usepackage{array}
\usepackage[colorlinks]{hyperref}

\begin{document} 

\maketitle

\begin{abstract}
Let $p$ be a prime. We say that a pro-$p$ group is self-similar of index $p^k$ 
if it admits a faithful self-similar action on a $p^k$-ary regular rooted tree
such that the action is transitive on the first level. The self-similarity index  of a self-similar pro-$p$ group $G$ is defined to be the least power of $p$, say $p^k$,  such that $G$ is self-similar of index $p^k$.
 We show that for every prime $p\ges 3$ and all integers $d$  there exist infinitely many
pairwise non-isomorphic self-similar  3-dimensional hereditarily just-infinite uniform  pro-$p$ groups of
 self-similarity index greater than $d$. 
 This implies that, in general, for self-similar $p$-adic analytic pro-$p$ groups
one cannot bound the self-similarity index 
 by a function that depends only on the dimension of the group.

\end{abstract}
{
\let\thefootnote\relax
\let\thefootnote\relax
}



\section*{Introduction}

The class of groups that admit faithful self-similar actions on regular 
rooted trees  contains many important and interesting examples, 
such as the Grigorchuk 
2-group \cite{Gri80}, the Gupta-Sidki $p$-groups \cite{GuSi83}, 
and
groups obtained as iterated monodromy groups of self-coverings 
of the Riemann sphere by post-critically finite rational maps \cite{NekSSgrp}.
Over the last 15 years there has been an intensive study on the self-similar actions of other 
important families of groups including 
finitely generated nilpotent groups \cite{BeSi07},  arithmetic groups \cite{Ka12}, and 
groups of type $\mr{FP}_n$ \cite{KoSi20}. 
Self-similar actions of some classes of finite $p$-groups were studied 
in \cite{Su11} and \cite{BaFaFeVa20}.  
For the definition of self-similar action, as well as for other examples,
the reader may consult, for instance, \cite{NekSSgrp}.
We say that a group $G$ is \textbf{self-similar of index} $\bm{d}$ 
if $G$ admits a faithful self-similar action on
a $d$-ary regular rooted tree
such that the action is transitive on the first level; moreover,
we say that $G$ is \textbf{self-similar} if it is self-similar of 
some index $d$. The \textbf{self-similarity index} $\sigma(G)$  of a self-similar group $G$ 
is defined to be the least positive integer $d$ such that $G$ is self-similar of index $d$.
In case $G$ is not self-similar we put $\sigma(G)=\infty$.

Throughout the paper, let $p$ be a prime. Observe that if a pro-$p$ group is self-similar of
index $d$ then $d$ is a power of $p$.
In \cite{NS2019} 
we initiated the study of self-similar actions of $p$-adic analytic pro-$p$ groups. 
For the rest of this paragraph, assume that $p\ges 5$.
In \cite{NS2019} we classified the 
3-dimensional \textit{unsolvable} torsion-free $p$-adic analytic pro-$p$ groups,
and we determined which of them are self-similar of index $p$. 
On the other hand, in \cite{NShered} we determined which 3-dimensional \textit{solvable}
torsion-free $p$-adic analytic pro-$p$ groups are self-similar of index $p$. Moreover, 
we established that every 3-dimensional \textit{solvable}
torsion-free $p$-adic analytic pro-$p$ group
is self-similar of index $p^2$; as a consequence, the self-similarity index of any such group
is less or equal to $p^2$. 
It is worth mentioning that the results of \cite{NS2019}
do not exclude the possibility that the self-similarity index of \textit{self-similar} 
3-dimensional unsolvable torsion-free $p$-adic analytic pro-$p$ groups
is bounded; indeed, they do not even exclude the possibility that this bound is $p^2$. 

\medskip 

In this short note we consider the following question.

\medskip

\noindent \textbf{Question (*).}
Is the self-similarity index of a \textit{self-similar} $p$-adic analytic pro-$p$ group 
bounded by a function that only depends on the dimension?
\\

\vspace{0mm}
\noindent

Let $G$ be a torsion-free unsolvable $p$-adic analytic pro-$p$ group of dimension 3. 
Then $G$ has finite abelianization, so there exist unique
numbers
$s_0,s_1,s_2\in \bb{N}$ such that $s_0 \les s_1 \les s_2$ and
$G/[G,G]\simeq \bigoplus_{i=0}^2 (\bb{Z}/p^{s_i}\bb{Z})$; these numbers are  called the $\bm{s}$\textbf{-invariants} 
of $G$. 
For such $G$ we define $k(G)\in\bb{N}$ by
$$  k(G):= \left\lceil \mr{min}\left( \frac{s_1-s_0}{2},\frac{s_2-s_1}{2} \right)\right\rceil. $$
The following theorem, which is the main result of this paper, improves significantly
the lower bound for the self-similarity indices
of 3-dimensional unsolvable 
torsion-free
$p$-adic analytic pro-$p$ groups given in \cite{NS2019}.
The main ingredient in the proof of this theorem is an analogous result for $\bb{Z}_p$-Lie lattices
(Theorem \ref{nsshigh}) which is interesting on its own. 

\begin{theorem}\label{main-group}
Let $p\ges 3$ be a prime, and let $G$ be a saturable unsolvable $p$-adic analytic pro-$p$ group of dimension 3. 
Then $\sigma(G) \ges p^{k(G)}$.
\end{theorem}

\noindent 
A finitely generated pro-$p$ group is saturable if it admits 
a certain type of valuation map; 
for precise details we refer to \cite[Section 3]{GSpsat}.
Saturable groups, which were introduced by Lazard \cite{Laz65}, 
play a central role in the theory of $p$-adic analytic groups:
a topological group
is $p$-adic  analytic if and only if it contains an open finitely generated
pro-$p$ subgroup which is saturable \cite[Sections III(3.1) and III(3.2)]{Laz65}.

\medskip 

The following corollary is a direct consequence of Theorem \ref{main-group} 
and the fact that  any torsion-free $p$-adic analytic pro-$p$ group $G$ of dimension
$\mr{dim}(G)<p$ is saturable (see \cite[Theorem A]{GSKpsdimJGT}).

\begin{corollary}
Let $p\ges 5$ be a prime, and let $G$ be a torsion-free unsolvable $p$-adic analytic pro-$p$ group of dimension 3. 
Then $\sigma(G) \ges p^{k(G)}$.
\end{corollary}

We denote by $SL_2^{\sylow}(\bb{Z}_p)$ and $SL_2^{1}(\bb{Z}_p)$, respectively,
a Sylow pro-$p$ subgroup and the first congruence subgroup of $SL_2^{}(\bb{Z}_p)$.
The next corollary is a consequence of \cite[Theorem C]{NS2019} and Theorem \ref{main-group}.

\begin{corollary}
Let $p$ be a prime, and
let either $p\ges 5$ and $G$ be an open subgroup of $SL_2^{\sylow}(\bb{Z}_p)$,
or $p\ges 3$ and $G$ be an open subgroup of $SL_2^{1}(\bb{Z}_p)$.
Then $G$ is self-similar of self-similarity index $\sigma(G) \ges p^{k(G)}$.
\end{corollary}

A pro-$p$ group is said to be powerful if $[G,G]\les G^{p^\vep}$,
where $\vep=1$ when $p\ges 3$,
$\vep = 2$ when $p=2$, and $G^{p^\vep}$ is the closure of the subgroup of $G$ generated
by the ${p^\vep}$-th powers of the elements of $G$. A powerful pro-$p$ group is called uniform 
if it is finitely generated and torsion-free.
We point out that uniform pro-$p$ groups are saturable.
A pro-$p$ group $G$ is said to be just infinite if it is infinite and any nontrivial normal closed subgroup
of $G$ has finite index in $G$; moreover, if every open subgroup of $G$ is just infinite
then $G$ is called hereditarily just infinite.

\medskip

The following corollary provides a negative answer to Question (*).

\begin{corollary}\label{cor4} 
For all primes $p\ges 3$ and all integers $d$, there exist infinitely many
pairwise non-isomorphic self-similar 3-dimensional hereditarily just-infinite 
uniform pro-$p$ groups of
self-similarity index greater than $d$.
\end{corollary}

We close the Introduction by observing that Theorem \ref{main-group} provides
further evidence for \cite[Conjecture E]{NS2019}.

\vspace{5mm}

\noindent
\textbf{Notation.} The set $\bb{N}=\{0,1,2,...\}$ of natural numbers is assumed to contain $0$. 
We denote by $v_p$ the $p$-adic valuation on the field $\bb{Q}_p$ of $p$-adic numbers.
A $\bb{Z}_p$-Lie lattice is a Lie algebra over the ring $\bb{Z}_p$ of $p$-adic integers
the underlying $\bb{Z}_p$-module of which is finitely generated and free.
The submodule generated by the elements $x_1,...,x_n$ of  some $\bb{Z}_p$-module
will be denoted by $\gen{x_1,...,x_n}$.

\section{Results on Lie lattices}

This section contains the main technical result of the paper, Theorem \ref{nsshigh},
which is a result on non-self-similarity of $\bb{Z}_p$-Lie lattices. For the ease 
of the reader, we recall some definitions and notations (see \cite{NS2019}
for more details).
Let $L$ be a $\bb{Z}_p$-Lie lattice, and let $k\in\bb{N}$.

A virtual endomorphism of $L$ of index $p^k$ is an algebra morphism $\varphi:M\rar L$
where $M$ is a subalgebra of $L$ of index $p^k$.
An ideal $I$ of $L$ is said to be $\varphi$-invariant if
$I\subseteq M$ and $\varphi(I)\subseteq I$.
We say that a virtual endomorphism $\varphi$ is simple if there are no nonzero ideals $I$
of $L$ that are $\varphi$-invariant.

We say that $L$ is \textbf{self-similar of index} $\bm{p^k}$ if there exists
a simple virtual endomorphism of $L$ of index $p^k$.
A \textbf{self-similar} $\bb{Z}_p$-Lie lattice is a $\bb{Z}_p$-Lie lattice which is 
self-similar of index $p^l$ for some $l\in\bb{N}$.
In case $L$ is self-similar we denote the minimum index of self-similarity 
by $\sigma(L)$, and we call it the \textbf{self-similarity index} of $L$.
In case $L$ is not self-similar we put $\sigma(L)=\infty$. 

Assume that  $L$ is \textit{3-dimensional} and \textit{unsolvable}.
Then $L$ has finite abelianization, so there exist unique numbers
$s_0,s_1,s_2\in \bb{N}$ such that $s_0 \les s_1 \les s_2$ and
$L/[L,L]\simeq \bigoplus_{i=0}^2 (\bb{Z}/p^{s_i}\bb{Z})$; 
these numbers are  called the $\bm{s}$\textbf{-invariants} 
of $L$. 
Recall from \cite[Definition 2.4]{NS2019} 
that a basis $(x_0,x_1,x_2)$ of $L$ is called a \textbf{diagonalizing basis}
if $[x_{i},x_{i+1}]=a_{i+2}x_{i+2}$ for some $a_0,a_1,a_2\in\bb{Z}_p$, where
the index $i\in\{0,1,2\}$ is interpreted modulo 3; moreover, such a basis is called 
well diagonalizing
if $v_p(a_0)\les v_p(a_1)\les v_p(a_2)$, in which case $s_i:=v_p(a_i)$, $i\in\{0,1,2\}$, 
are the $s$-invariants of $L$. 
We also recall that if $L$ admits a diagonalizing basis then it admits a 
well diagonalizing one (see \cite[Remark 2.5]{NS2019}).

\begin{theorem}\label{nsshigh}
Let $p$ be a prime, let $L$ be a 3-dimensional unsolvable $\bb{Z}_p$-Lie lattice,
and let $s_0\les s_1\les s_2$ be the $s$-invariants of $L$.
Assume that $L$ admits a diagonalizing basis,
and define
$$ K = \left\lceil \mr{min}\left( \frac{s_1-s_0}{2},\frac{s_2-s_1}{2} \right)\right\rceil. $$
Then $\sigma(L) \ges p^K$.
\end{theorem}

\begin{remark} \label{rdiag}
Recall from \cite[Proposition 2.7]{NS2019}
that if $p\ges 3$ then any 3-dimensional unsolvable $\bb{Z}_p$-Lie lattice
admits a diagonalizing basis.
\end{remark}

Before proving Theorem \ref{nsshigh}, we apply it to construct a
family of self-similar 3-dimensional $\bb{Z}_p$-Lie lattices with
unbounded self-similarity index. A $\bb{Z}_p$-Lie lattice $L$ is called
powerful if $[L,L]\subseteq p^{\vep}L$, 
where $\vep=1$ when $p\ges 3$,
and $\vep = 2$ when $p=2$.

\begin{example}\label{efamily}
We construct a sequence $L_l$, $l\in\bb{N}$, of self-similar 
powerful 3-dimensional unsolvable $\bb{Z}_p$-Lie lattices
such that $\sigma_l\rar \infty$ for $l\rar \infty$,
where $\sigma_l$ is the self-similarity index of $L_l$.
As $\bb{Z}_p$-module, define $L_l=\bb{Z}_p^3$, and
let $(x_0,x_1,x_2)$ be the canonical basis of $\bb{Z}_p^3$.
Let the bracket of $L_l$ be induced by the commutation
relations 
$[x_1,x_2] = p^2 x_0$, 
$[x_2,x_0] = p^{2l+2}x_1$ and
$[x_0,x_1] = -p^{4l+2}x_2$.
Then $L_l$ is a powerful unsolvable 3-dimensional $\bb{Z}_p$-Lie lattice
with $s$-invariants $s_0=2$, $s_1=2l+2$ and $s_2=4l+2$.
From Theorem \ref{nsshigh} it follows that 
$\sigma_l \ges p^l$, hence, $\sigma_l\rar\infty$ for $l\rar \infty$.
It remains to prove that $L_l$ is self-similar, for which it is enough to
prove that there exists a self-similar finite-index subalgebra $M$ of $L$
(see \cite[Lemma 2.1]{NS2019}). Define 
$y_0=2x_0$, 
$y_1= p^lx_1-p^{2l}x_2$ and 
$y_2=p^lx_1+p^{2l}x_2$.
A straightforward computation gives
$[y_1,y_2] = p^{3l+2} y_0$, 
$[y_2,y_0] = 2p^{3l+2}y_2$ and
$[y_0,y_1] = 2p^{3l+2}y_1$.
Then $M:=\gen{y_0,y_1,y_2}$ is a finite-index subalgebra of $L$.
From \cite[Lemma 2.11]{NS2019} it follows that $M$ is self-similar, as desired.
\end{example}

\begin{corollary}\label{corlie}
For all primes $p$ and all integers $d$ there exist infinitely many
pairwise non-isomorphic self-similar powerful 3-dimensional unsolvable
$\bb{Z}_p$-Lie lattices with self-similarity index greater than $d$.
\end{corollary}

\begin{proof}
The result follows from Example \ref{efamily}.
\end{proof}

\vspace{5mm}

The following lemmas will be applied in the proof of Theorem \ref{nsshigh},
which is given at the end of the section. Recall that a 
$\bb{Z}_p$-Lie lattice $L$ 
is said to be just infinite if it is infinite and all
nonzero ideals of $L$ have finite index in $L$.
A $\bb{Z}_p$-Lie lattice is said to be
hereditarily just infinite if all of its finite-index subalgebras
are just infinite.

\begin{lemma}\label{l1}
Let $L$ be a 3-dimensional unsolvable $\bb{Z}_p$-Lie lattice
that admits a diagonalizing basis.
Then $L$ is hereditarily just infinite.
\end{lemma}

\begin{proof}
Let $M$ be a finite-index subalgebra of $L$. Clearly, $M$ is infinite.
Let $I$ be a nonzero ideal of $M$. We have to show that $\mr{dim}\,I=3$.
There exists $k\in\bb{N}$ such that $p^kL\subseteq M$.
Let $(x_0,x_1,x_2)$ be a diagonalizing basis of $L$,
and let $[x_i,x_{i+1}]=a_{i+2}x_{i+2}$,
where the index $i\in\{0,1,2\}$ is interpreted modulo 3,
and $a_0,a_1,a_2$ are nonzero elements of $\bb{Z}_p$. 
We will show that $I$ contains nonzero multiples
of $x_i$ for any $i\in \{0,1,2\}$. From this, it follows that $I$ is 3-dimensional,
as desired.
Let $y = b_0x_0+b_1x_1+b_2x_2$ be a nonzero element of
$I$, where $b_0$, $b_1$ and $b_2$ are elements of $\bb{Z}_p$. 
Without loss of generality we can assume that $b_0\neq 0$.
Since $I$ is an ideal of $M$, and $p^kx_i\in M$ for $i\in\{0,1,2\}$,
the commutators $z_1 := [[y,p^kx_1],p^kx_0]$ and $z_2 := [[y,p^kx_2],p^kx_0]$ belong to $I$.
One easily checks that $z_j$ is a nonzero multiple of $x_j$ for $j\in\{1,2\}$.
Hence, $z_0:=[z_1,z_2]$ is a nonzero multiple of $x_0$ and belongs
to $I$ as well.
\end{proof}

\begin{lemma}\label{l2}
Let $L$ be a hereditarily just-infinite $\bb{Z}_p$-Lie lattice,
and let $\varphi:M\rar L$ be a virtual endomorphism of $L$. 
If $\varphi$ is simple then $\varphi$ is injective.
\end{lemma}

\begin{proof}
By contrapositive, assume that $\varphi$ is not injective.
Since, by assumption, $M$ has finite index in $L$, $M$ is just infinite.
Since $\mr{ker}(\varphi)$ is a nonzero ideal of $M$, $\mr{ker}(\varphi)$
has finite index in $M$, hence, in $L$. Then there exists $k\in\bb{N}$
such that $p^kL\subseteq \mr{ker}(\varphi)$. Hence,
$p^kL$ is a nonzero (since $L$ is infinite) $\varphi$-invariant 
ideal of $L$. It follows that $\varphi$ is not simple.
\end{proof}
\\

The proof of the following lemma is straightforward, and it is left to the reader.

\begin{lemma}\label{lidcond}
Let $L$ be a 3-dimensional unsolvable $\bb{Z}_p$-Lie lattice.
Let $(x_0,x_1,x_2)$
be a basis of $L$ such that 
$[x_i,x_{i+1}]=a_{i+2}x_{i+2}$, where $a_i\in\bb{Z}_p\bs\{0\}$ and
the index $i\in\{0,1,2\}$ is interpreted modulo 3; define $s_i := v_p(a_i)$.
Let $m_i\in\bb{N}$ and define $I := \gen{p^{m_0}x_0, p^{m_1}x_1, p^{m_2}x_2}$.
Then $I$ is an ideal of $L$ if and only if $m_i+s_j\ges m_j$ for
all $i$ and $j$ in $\{0,1,2\}$.
\end{lemma}

\begin{lemma}\label{ldiag}
Let the index $i$ take values in $\{0,1,2\}$.
Let $a_i\in\bb{Q}_p\bs\{0\}$ and define $s_i=v_p(a_i)$.
Consider the following parameters:
$$ 
k_1,k_2\in\bb{N},
\qquad
e,f,g\in \bb{Z}_p,
\qquad
r_0,r_1\in\bb{Z},$$
make the following definitions:
\begin{eqnarray*}
h &:=& eg-fp^{k_1}\\
a_3 &:=& a_0p^{2k_1+2k_2}+a_1e^2p^{2k_2}+a_2h^2\\
a_4 &:=& a_0a_1p^{2k_2}+a_0a_2g^2+ a_1a_2f^2\\
m_0 &:=& r_1+s_0+k_1+k_2\\ 
m_1 &:=& s_0+s_1+k_1+2k_2\\
m_2 &:=& r_0+k_2,
\end{eqnarray*}
and define
the $3\times 3$ matrix $W$ with
coefficients in $\bb{Q}_p$ by
$$
W := 
\left[
\begin{array}{l|l|l}
a_0p^{r_1+k_1+k_2} & 
a_0(a_1ep^{2k_2}+a_2gh) & 
fp^{r_0} \\
-a_1e p^{r_1+k_2} & 
a_1(a_0p^{k_1+2k_2}-a_2fh) &
gp^{r_0} \\
a_2hp^{r_1}&
-a_2(a_0gp^{k_1}+a_1ef)p^{k_2}&
p^{r_0+k_2} \\
\end{array}
\right].
$$
Assume that
\begin{enumerate}
\item $s_0+s_1+k_1+2k_2\les r_0\les s_0+s_2+k_1$.
\item $s_0+k_1+k_2\les r_1\les s_1-k_1+k_2$.
\item $v_p(a_3) = s_0 +2k_1+2k_2$.
\item $v_p(a_4) = s_0 +s_1+2k_2$.
\end{enumerate}
Then there exists $V\in GL_3(\bb{Z}_p)$ such that 
$WV = \mr{diag}(p^{m_0}, p^{m_1}, p^{m_2})$.
If, moreover, $s_0\ges 0$ then $m_i+s_j\ges m_j$ for
all $i$ and $j$ in $\{0,1,2\}$.
\end{lemma}

\begin{proof}
Observe that $s_0+2k_1\les s_1$ and $s_1+2k_2\les s_2$.
Also, $s_0+2k_1+k_2\les r_0-r_1\les s_2-k_2$.
We have $v_p(W_{00})= m_0$, $v_p(W_{01})\ges s_0+s_1+2k_2$
and $v_p(W_{02})\ges r_0$. Hence,
$v_p(W_{01})\ges m_0$ and $v_p(W_{02})\ges m_0$.
It follows that there exists $V_1\in GL_3(\bb{Z}_p)$ such
that $W_1:=WV_1$ is equal to
$$
W_1 = 
\left[
\begin{array}{l|l|l}
a_0p^{r_1+k_1+k_2} & 
0 & 
0 \\
-a_1e p^{r_1+k_2} & 
a_1a_3p^{-k_1}&
a_0^{-1}(a_0gp^{k_1}+a_1ef)p^{r_0-k_1}\\
a_2hp^{r_1}&
-a_2a_3gp^{-k_1-k_2}&
a_0^{-1}(a_0p^{k_1+2k_2}-a_2fh)p^{r_0-k_1-k_2}
\end{array}
\right].
$$
We have $v_p((W_1)_{11})=m_1$, 
$v_p((W_1)_{10})\ges r_1+ s_1+k_2$ and $v_p((W_1)_{12})\ges r_0$.
Hence, $v_p((W_1)_{10})\ges m_1$ and $v_p((W_1)_{12})\ges m_1$.
It follows that there exists $V_2\in GL_3(\bb{Z}_p)$ such
that $W_2:=W_1V_2$ is equal to
$$
W_2 = 
\left[
\begin{array}{l|l|l}
a_0p^{r_1+k_1+k_2} & 
0 & 
0 \\
-a_1e p^{r_1+k_2} & 
a_1a_3p^{-k_1}&
0\\
a_2hp^{r_1}&
-a_2a_3gp^{-k_1-k_2}&
a_0^{-1}a_1^{-1}a_4p^{r_0-k_2}
\end{array}
\right].
$$
We have $v_p((W_2)_{22})=m_2$,
$v_p((W_2)_{20})\ges r_1+s_2$ and $v_p((W_2)_{21})\ges s_0+s_2+k_1+k_2$.
Hence, $v_p((W_2)_{20})\ges m_2$ and $v_p((W_2)_{21})\ges m_2$.
The claim about $V$ follows.

The proof of the claim $m_i + 	s_j \ges m_j$ is straightforward and it is left to the reader.
\end{proof}

\vspace{5mm}

\noindent 
\textbf{Proof of Theorem \ref{nsshigh}.}
Let $k\in 	\bb{N}$ be such that $k< K$. We prove that $L$ is not self-similar of index $p^k$,
from which the theorem will follow.

Observe that $s_0+2k< s_1$ and $s_1+2k< s_2$;
in particular, $s_0< s_1< s_2$. Let $x=(x_0,x_1,x_2)$ 
be a well diagonalizing basis of $L$, and 
let $[x_i,x_{i+1}]=a_{i+2}x_{i+2}$, where $a_i\in\bb{Z}_p\bs\{0\}$ and
the index $i\in\{0,1,2\}$ is interpreted modulo 3. Then $s_i = v_p(a_i)$.
 
Let $M$ be a subalgebra of $L$ of index $p^k$, and 
let $\varphi: M\rar L$ be an algebra morphism. We have to show that
$\varphi$ is not simple. 
By Lemma \ref{l1} and Lemma \ref{l2} we can assume that $\varphi$ is injective.

We will exhibit a nonzero $\varphi$-invariant ideal $I$ of $L$,
from which it follows that $\varphi$ is not simple.
Let $M'=\varphi(M)$.
Then $M'$ is a subalgebra of $L$ of index $p^k$
(see \cite[Proposition 2.8]{NS2019}), and $\varphi$ induces an
isomorphism $M\rar M'$.
There exist $k_0,k_1,k_2,k_0',k_1',k_2'\in\bb{N}$
and $e,f,g,e',f',g'\in\bb{Z}_p$ such that 
$M= \gen{y_0,y_1,y_2}$ and $M'=\gen{y_0',y_1',y_2'}$,
where $y_j = \sum_i U_{ij}x_i$, 
$y_j' = \sum_i U_{ij}'x_i$ and
$$
U = 
\left[
\begin{array}{lll}
p^{k_0} & e & f \\
0 & p^{k_1} & g \\
0 & 0 & p^{k_2}\end{array}
\right]
\qquad\qquad
U' = 
\left[
\begin{array}{lll}
p^{k_0'} & e' & f' \\
0 & p^{k_1'} & g' \\
0 & 0 & p^{k_2'}\end{array}
\right].
$$
Note that $k_0+k_1+k_2 = k_0'+k_1'+k_2' =k$.

In this paragraph we consider the subalgebra $M$.
Define $a_3 = a_0p^{2k_1+2k_2}+a_1e^2p^{2k_2}+a_2h^2$
and $a_4 = a_0a_1p^{2k_2}+a_0a_2g^2+ a_1a_2f^2$,
where $h=eg-fp^{k_1}$.
From the assumptions we deduce 
$v_p(a_3) = s_0 +2k_1+2k_2$ and $v_p(a_4) = s_0 +s_1+2k_2$;
in particular $a_3$ and $a_4$ are nonzero.
It is not difficult to check that the entries of the matrix
$$ 
V = 
\left[
\begin{array}{l|l|l}
1 &
0 & 
0 \\
-a_3^{-1}(a_1ep^{2k_2}+a_2gh)p^{k_0}&
1&
0\\
a_3^{-1}a_2hp^{k_0+k_1}&
-a_{4}^{-1}a_2(a_0gp^{k_1}+a_1ef) & 
1
\end{array}
\right]
$$
are in $\bb{Z}_p$; hence, $V\in GL_3(\bb{Z}_p)$.
Define $z=(z_0,z_1,z_2)$ by $z_j = \sum_i V_{ij}y_i$;
hence, $z$ is a basis of $M$.
We have $z_j = \sum_i (UV)_{ij}x_i$,
and one can compute $UV$ and obtain
$$ 
UV = 
\left[
\begin{array}{l|l|l}
a_3^{-1}a_0p^{k_0+2k_1+2k_2} & 
a_{4}^{-1}a_0(a_1ep^{2k_2}+a_2gh) & 
f \\
-a_3^{-1}a_1ep^{k_0+k_1+2k_2} & 
a_{4}^{-1}a_1(a_0p^{k_1+2k_2}-a_2fh) & 
g\\
a_3^{-1}a_2hp^{k_0+k_1+k_2} & 
-a_{4}^{-1}a_2(a_0gp^{k_1}+a_1ef)p^{k_2} & 
p^{k_2}
\end{array}
\right].
$$
A straightforward but lengthy calculation yields that   
$[z_i,z_{i+1}]=c_{i+2}z_{i+2}$,
where 
$$
c_0 = a_3p^{-k},
\qquad
c_1 = a_3^{-1}a_4p^k,
\qquad
c_2 = a_0a_1a_2a_4^{-1}p^k.
$$
Observe that $c_i\in\bb{Z}_p$ since, by assumption, $M$ is a subalgebra of $L$.
Hence, $z$ is a diagonalizing basis of $M$, and 
the $p$-adic valuations $t_i:=v_p(c_i)$ are the $s$-invariants of $M$.
One can compute that $t_i = s_i+k-2k_i$. Observe that $0\les t_0<t_1< t_2$;
in particular, $z$ is a well diagonalizing basis of $M$.
Let $a_3 = u_3p^{s_0+2k_1+2k_2}$ and $a_4=u_4p^{s_0+s_1+2k_2}$
with $u_3,u_4\in\bb{Z}_p^*$. 
In order to get a matrix in a slightly simplified form,
we multiply the columns 0 and 1 of $UV$
by $u_3$ and $u_4$, respectively, getting the matrix
$$ 
W =
\left[
\begin{array}{l|l|l}
a_0p^{-s_0+k_0} & 
a_0(a_1ep^{2k_2}+a_2gh)p^{-s_0-s_1-2k_2} & 
f \\
-a_1ep^{-s_0+k_0-k_1} & 
a_1(a_0p^{k_1+2k_2}-a_2fh)p^{-s_0-s_1-2k_2} & 
g\\
a_2hp^{-s_0+k_0-k_1-k_2} & 
-a_2(a_0gp^{k_1}+a_1ef)p^{-s_0-s_1-k_2} & 
p^{k_2}
\end{array}
\right].
$$
Define $w = (w_0,w_1,w_2)$ by $w_j =\sum_i W_{ij}x_i$, and
note that $w$ is a well diagonalizing basis of $M$.
We have
$$ 
[M,M] = \gen{p^{t_0}w_0, p^{t_1}w_1,p^{t_2}w_2}
$$
and (see \cite[Lemma 2.13]{NS2019})
$$ 
\gamma_2(M) = \gen{p^{t_0+t_1}w_0, p^{t_0+t_1}w_1,p^{t_0+ t_2}w_2}.
$$
Let $r_0: = s_0+s_1+k_1+2k_2$ and $r_1 := s_0+k_1+k_2$, and define
$$I := p^{r_0}M+p^{r_1}[M,M]+\gamma_2(M).$$
We have
$I =\gen{p^{r_1+t_0}w_0, p^{t_0+t_1}w_1, p^{r_0}w_2}$
and $I\subseteq M$.
One way to prove the last equality is to observe that 
$r_0=t_0+t_1+k_1\ges t_0+t_1$ and $r_1 = t_0+k_0\ges t_0$,
and to check that the following inequalities hold:
$r_1+t_0\les r_0$, $r_1+t_0\les t_0+t_1$,
$t_0+t_1\les r_0$, $t_0+t_1\les r_1+t_1$,
$r_0\les r_1+t_2$ and $r_0\les t_0+t_2$.
Let $\widetilde{W}$ be the matrix the columns of which
are the coordinates with respect to $x$ of 
$p^{r_1+t_0}w_0$, $p^{t_0+t_1}w_1$ and $p^{r_0}w_2$, 
respectively. We have 
$$
\widetilde{W} = 
\left[
\begin{array}{l|l|l}
a_0p^{r_1+k_1+k_2} & 
a_0(a_1ep^{2k_2}+a_2gh) & 
fp^{r_0} \\
-a_1e p^{r_1+k_2} & 
a_1(a_0p^{k_1+2k_2}-a_2fh) &
gp^{r_0} \\
a_2hp^{r_1}&
-a_2(a_0gp^{k_1}+a_1ef)p^{k_2}&
p^{r_0+k_2} \\
\end{array}
\right].
$$
From Lemma \ref{ldiag} and Lemma \ref{lidcond}
it follows that 
$$I = \gen{p^{m_0}x_0, p^{m_1}x_1, p^{m_2}x_2}$$
and that $I$ is an ideal of $L$,
where $m_0 = r_1+s_0+k_1+k_2$, $m_1 = s_0+s_1+k_1+2k_2$
and $m_2 = r_0+k_2$.
Since $I$ is nonzero, 
 it suffices to show
that $I$ is $\varphi$-invariant. We know that $I\subseteq M$,
and we will show that $\varphi(I) = I$, which will
complete the  proof
of the theorem.

The same arguments as in the preceeding paragraph may be applied to
$M'$. In particular, the $s$-invariants of $M'$
are $t_i' = s_i+k-2k_i'$ and $0\les t_0'<t_1'<t_2'$.
Since $M'\simeq M$, the two subalgebras have the same $s$-invariants,
and because the $t_i$'s and the $t_i'$'s are ordered, 
we deduce that $t_i'=t_i$ for $i\in\{0, 1, 2\}$. Thus, $k_i'=k_i$ for $i\in\{0, 1, 2\}$.
Let 
$$I' := p^{r_0}M'+p^{r_1}[M',M']+\gamma_2(M').$$
Clearly, $\varphi(I) = I'$. Also, again from the arguments
of the preceeding paragraph applied to $M'$ (and from
$k_i'=k_i$), we have $I'= \gen{p^{m_0}x_0, p^{m_1}x_1, p^{m_2}x_2} =I$. 
Hence, $\varphi(I)=I$, as desired.
\ep

\section{Proof of the main theorem}

In \cite{Laz65} Lazard constructed  a pair of mutually inverse functors
between the category of saturable $p$-adic analytic pro-$p$ groups and the
category of saturable $\bb{Z}_p$-Lie lattices (see \cite[IV(3.2.6)]{Laz65};
see also \cite{DixAnaProP} and \cite{GSKpsdimJGT}). 
These functors commute with the forgetful
functors to the category of sets, and we refer to the isomorphism 
of categories that they define as Lazard's correspondence. 
Whenever $G$ is a saturable $p$-adic analytic pro-$p$ group, 
we denote by $L_G$ the associated saturable $\bb{Z}_p$-Lie lattice. 
Under Lazard's correspondence, uniform pro-$p$ groups correspond to powerful $\bb{Z}_p$-Lie lattices.

\begin{lemma}\label{lsinv}
Let $G$ be a saturable $p$-adic analytic pro-$p$ group, and 
let $L_G$ be the associated $\bb{Z}_p$-Lie lattice. 
Then $G/[G,G]=L_G/{[L_G,L_G]}$ as abelian groups.
\end{lemma}

\begin{proof}
By \cite[Theorem A and Proposition 2.1(3)]{GSpsat} 
we have that $[G,G]$ is PF-embedded
in $G$, and by \cite[Theorem B(4)]{GSpsat} we have $[G,G]=[L_G,L_G]$. Moreover, by \cite[Theorem 4.5]{GSpsat}
we have $[G,G]=L_{[G,G]}$ as sets, and $G/[G,G]=L_G/L_{[G,G]}$ as abelian groups.
Then $G/[G,G]=L_G/{[L_G,L_G]}$ as abelian groups.
\end{proof}
\\

\noindent
\textbf{Proof of Theorem \ref{main-group}.} 
Let $L_G$ be the $\bb{Z}_p$-Lie lattice associated with $G$,
and observe that $L_G$ has dimension 3.
By \cite[Proposition A]{NS2019} we have $\sigma(G)=\sigma(L_G)$,
and by \cite[Theorem B(4)]{GSpsat} we have that $L_G$ is unsolvable. From Remark \ref{rdiag}
we see that $L_G$ admits a diagonalizing basis,
and from Lemma \ref{lsinv} we see that the $s$-invariants of $G$ coincide with the ones 
of $L_G$, hence, the number $K$ defined from $L_G$ in the statement of 
Theorem \ref{nsshigh} is equal to $k(G)$. 
By applying Theorem \ref{nsshigh} we deduce that $\sigma(L_G)\ges p^{K}$,
and the desired conclusion $\sigma(G)\ges p^{k(G)}$ follows.
\\

\noindent
\textbf{Proof of Corollary \ref{cor4}.} 
For $l\in\bb{N}$, let $L_l$ and $\sigma_l$ be defined as in Example \ref{efamily},
and recall that $\sigma_l\rar \infty$ for $l\rar \infty$.
Let $G_l$ be the uniform pro-$p$ group associated with $L_l$.
Then, by \cite[Proposition A and Theorem 3.1]{NS2019}, $G_l$ is a self-similar 3-dimensional 
hereditarily just-infinite
uniform pro-$p$ group of self-similarity index $\sigma_l$.
The corollary follows.


\begin{footnotesize}

\providecommand{\bysame}{\leavevmode\hbox to3em{\hrulefill}\thinspace}
\providecommand{\MR}{\relax\ifhmode\unskip\space\fi MR }
\providecommand{\MRhref}[2]{%
  \href{http://www.ams.org/mathscinet-getitem?mr=#1}{#2}
}
\providecommand{\href}[2]{#2}

\end{footnotesize}

\vspace{5mm}

\noindent
Francesco Noseda\\
Mathematics Institute\\
Federal University of Rio de Janeiro\\
Av. Athos da Silveira Ramos, 149\\
21941-909, Rio de Janeiro, RJ\\
Brazil  \\
{\tt noseda@im.ufrj.br}\\

\noindent
Ilir Snopce \\
Mathematics Institute\\
Federal University of Rio de Janeiro\\
Av. Athos da Silveira Ramos, 149\\
21941-909, Rio de Janeiro, RJ\\
Brazil  \\
{\tt ilir@im.ufrj.br}\\

\end{document}